\documentclass[12pt]{article}
\usepackage{amsmath}
\usepackage{graphicx}
\usepackage{enumerate}
\usepackage{natbib}
\usepackage{url} 
\usepackage[colorlinks,citecolor=blue,urlcolor=blue]{hyperref}
\usepackage{amsthm}
\usepackage{amsfonts}
\usepackage{bm}
\usepackage{booktabs}
\usepackage{multirow}
\usepackage[justification=centering]{caption}
\usepackage[plain,noend]{algorithm2e}
\usepackage{threeparttable}
\usepackage{longtable}
\usepackage{caption}
\usepackage{subcaption}
\usepackage{float}
\usepackage{chngcntr}
\usepackage{comment}
\usepackage{threeparttable}
\usepackage{cleveref}

\newcommand{\blind}{0}

\newtheorem{theorem}{Theorem}
\newtheorem{proposition}{Proposition}
\newtheorem{condition}{Condition}
\newtheorem{definition}{Definition}

\newtheorem{remark}{Remark}

\newcommand{\yi}[1][z]{Y_i(#1)}
\newcommand{\yj}[1][z]{Y_j(#1)}
\newcommand{\ymbar}[1][z]{\bar Y_{[m]}(#1)}
\newcommand{\ymbarhat}[1][z]{\bar Y_{[m] #1}}
\newcommand{\yimp}[1][z]{Y_i^{\mathrm{imp}}(#1)}

\def \yo{Y_i}

\newcommand{\varm}[1][]{S^2_{[m]Y(#1)}}
\def \varmdif{S^2_{[m]\tau}}
\newcommand{\varmhat}[1][]{s^2_{[m]Y(#1)}}

\def \covm{S_{[m]Y(1),Y(0)}}

\def \covmh{S_{[m]Y(1),Y(0)}^U}
\def \covmhhat{s_{[m]Y(1),Y(0)}^U}

\newcommand{\nm}[1][]{n_{[m] #1}}

\def \taub{\tau^*}

\def \taubhat{\hat \tau^*}

\def \yob{Y_i^*}

\def \taum{\tau_{[m]}}
\def \taumhat{\hat \tau_{[m]}}
\def \pim{\pi_{[m]}}
\def \rhom{\rho_{[m]}}

\def \sigmaAGL{\hat \sigma_S}

\def \sigmaneyman{\hat \sigma_N}

\def\asyequi{\xrightarrow[]{d}}

\def \i{\mathrm i}
\def \p{\mathrm{pair}}

\def \W{W}

\def \V{V}

\newcommand{\Gm}[1][y]{G_{[m]}(#1)}
\newcommand{\Gmhat}[1][y]{\hat G_{[m]}(#1)}
\newcommand{\Gminv}[1][u]{G_{[m]}^{-1}(#1)}
\newcommand{\Gminvhat}[1][u]{\hat G_{[m]}^{-1}(#1)}
\newcommand{\Fm}[1][y]{F_{[m]}(#1)}
\newcommand{\Fmhat}[1][y]{\hat F_{[m]}(#1)}
\newcommand{\Fminv}[1][u]{F_{[m]}^{-1}(#1)}
\newcommand{\Fminvhat}[1][u]{\hat F_{[m]}^{-1}(#1)}
\newcommand{\Hm}[1][y_1,y_0]{H_{[m]}(#1)}

\newcommand{\HmH}[1][y_1,y_0]{H_{[m]}^U(#1)}
\newcommand{\HmHhat}[1][y_1,y_0]{\hat H_{[m]}^U(#1)}
\newcommand{\HH}[1][y_1,y_0]{H^U(#1)}

\newcommand{\yma}[1][z]{Y_{[m]1}(#1)}
\newcommand{\ymb}[1][z]{Y_{[m]2}(#1)}
\newcommand{\ymaimp}[1][z]{Y_{[m]1}^{\mathrm{imp}}(#1)}
\newcommand{\ymbimp}[1][z]{Y_{[m]2}^{\mathrm{imp}}(#1)}
\newcommand{\ymiimp}[1][z]{Y_{[m]i}^{\mathrm{imp}}(#1)}
\newcommand{\ymao}[1][z]{Y_{[m]1}}
\newcommand{\ymbo}[1][z]{Y_{[m]2}}
\def \zm{Z_{[m]1}}

\usepackage{xcolor}

\begin{document}

\def\spacingset#1{\renewcommand{\baselinestretch}%
{#1}\small\normalsize} \spacingset{1}


\if1\blind
{
  \title{\bf }
  \author{Author 1\thanks{
    The authors gratefully acknowledge \textit{please remember to list all relevant funding sources in the unblinded version}}\hspace{.2cm}\\
    Department of YYY, University of XXX\\
    and \\
    Author 2 \\
    Department of ZZZ, University of WWW}
  \maketitle
} \fi

\if0\blind
{
  \title{\bf Sharp variance estimator and causal bootstrap in stratified randomized experiments}
  \author{Haoyang Yu \\
    Department of Statistics and Data Science, Tsinghua University\\
    Ke Zhu\thanks{Corresponding author: kzhu24@ncsu.edu.} \\
    Department of Statistics, North Carolina State University\\
    Department of Biostatistics and Bioinformatics, Duke University\\
    Hanzhong Liu\thanks{
    Dr. Liu was supported by the National Natural Science Foundation of China (12071242).}\\
    Department of Statistics and Data Science, Tsinghua University}
    \date{}
  \maketitle
} \fi

\bigskip
\begin{abstract}
Randomized experiments are the gold standard for estimating treatment effects, and randomization serves as a reasoned basis for inference.
In widely used stratified randomized experiments, randomization-based finite-population asymptotic theory enables valid inference for the average treatment effect, relying on normal approximation and a Neyman-type conservative variance estimator.
However, when the sample size is small or the outcomes are skewed, the Neyman-type variance estimator may become overly conservative, and the normal approximation can fail.
To address these issues, we propose a sharp variance estimator and two causal bootstrap methods to more accurately approximate the sampling distribution of the weighted difference-in-means estimator in stratified randomized experiments. 
The first causal bootstrap procedure is based on rank-preserving imputation and we prove its second-order refinement over normal approximation. The second causal bootstrap procedure is based on constant-treatment-effect imputation and is further applicable in paired experiments. 
In contrast to traditional bootstrap methods, where randomness originates from hypothetical super-population sampling, our analysis for the proposed causal bootstrap is randomization-based, relying solely on the randomness of treatment assignment in randomized experiments.
Numerical studies and two real data applications demonstrate advantages of our proposed methods in finite samples.
The \texttt{R} package \texttt{CausalBootstrap} implementing our method is publicly available.
\end{abstract}

\noindent%
\vfill

\newpage
\spacingset{1.9} 
\section{Introduction}
\label{s:intro}

Stratified randomized experiments have been widely employed in biomedical science, economics and social science to investigate the effects of an intervention. 
Compared to completely randomized experiments, stratification or blocking can improve estimation efficiency by balancing key pre-treatment covariates \citep{Fisher1926}. 
Recently, stratified randomized experiments have been extensively examined under the Neyman--Rubin potential outcomes model \citep{Rubin:1974,splawa1990application} and the randomization-based, or design-based finite-population, inference framework \citep{imbens2015causal,athey2017econometrics, mukerjee2018,fogarty2018,Liu2019,Pashley2017,schochet2022design,Wang2020rerandomization}. 
In the randomization-based inference framework, the randomized treatment assignment serves as the sole basis for statistical inference, a principle advocated by both \citet{fisher1935} and \citet{splawa1990application} for analyzing randomized experiments. This approach has recently garnered renewed attention in clinical trial analysis \citep{imai2008pair,rosenberger2015randomization,proschan2019re,rosenberger2019randomization,wang2020randomization,wang2020randomization2,zhu2023pair,heussen2024randomization,carter2024regulatory} due to its assumption-lean nature. Specifically, its validity relies on randomization, which is typically controlled, rather than on independent and identically distributed sampling from a super-population--an often unrealistic condition given the nonrandom nature of recruitment. Additionally, this framework does not depend on model assumptions, making it robust to model misspecification. Furthermore, randomization-based finite-population inference can be connected with super-population inference \citep{ding2017bridging}, allowing the two approaches to complement each other.
In the randomization-based framework, the sample average treatment effect (SATE) can be inferred using finite-population central limit theory alongside the Neyman-type variance estimator \citep{imbens2015causal,li2016,Liu2019}. However, this variance estimator tends to be conservative, resulting in wide confidence intervals. Furthermore, the asymptotic normal approximation may not be valid for small to moderate sample sizes, which are typical in high-cost randomized experiments.

Our first contribution is the proposal of a sharp variance estimator for the weighted difference-in-means in stratified randomized experiments, along with the establishment of its theoretical properties.
We construct the sharp variance estimator applying the sharp variance estimator introduced by \citet{aronow2014} within each stratum.
While the methodological extension is straightforward, handling stratified randomized experiments with both small and large strata poses a theoretical challenge.
We establish the formal theoretical results for the proposed sharp variance estimator, allowing either the number or size of strata to tend to infinity. 
The key technical difference between our work and the existing literature is the derivation of a concentration inequality for the stratum-size-weighted empirical distribution function, utilizing a fundamental inequality for sampling without replacement \citep{bobkov2004}.
We then prove a Glivenko--Cantelli type theorem for stratified randomized experiments, determining the convergence rate for the supremum of the difference between the empirical and population distribution functions.
Using the proposed sharp variance estimator, we construct the asymptotically narrowest valid Wald-type confidence interval for the average treatment effect (ATE) in these experiments.

Normal approximation combined with the sharp variance estimator can yield anti-conservative results in finite-sample simulations, particularly when the potential outcomes exhibit co-monotonicity \citep{aronow2014}. To enhance finite-sample performance, we propose a causal bootstrap method.
The bootstrap method \citep{Efron1979} is widely recognized for achieving second-order refinement in various statistical models \citep{hall2013bootstrap,wang2022bootstrap}.
Existing bootstrap methods in causal inference, however, typically operate within the super-population framework, where randomness arises from both sampling processes and treatment assignments \citep{shao2010theory,shao2013validity,zhang2020quantile,jiang2024bootstrap,kosko2024fast}.
In finite-population settings, where the sole source of randomness is the treatment assignment, the effectiveness of these methods remains unclear. To further explore this scenario, \citet{imbens2018causal} introduced a design-based causal bootstrap under the finite-population framework. Their approach, based on rank-preserving imputation for unobserved potential outcomes, established theoretical validity and second-order refinement.
Despite these advancements, the theoretical properties of the causal bootstrap are not fully understood in stratified randomized experiments. 
Moreover, the causal bootstrap based on rank-preserving imputation does not work in paired experiments.

Our second contribution is the establishment of consistency and second-order refinement for the causal bootstrap in stratified randomized experiments. 
We apply rank-preserving imputation within each stratum to generate a bootstrap distribution for pivotal statistics through stratified randomization. 
We then provide a theoretical guarantee for the causal bootstrap based on rank-preserving imputation, accommodating cases where both the number of strata and stratum sizes tend to infinity. 
Furthermore, by leveraging the theoretical insights from finite-population Edgeworth expansion \citep{babu1985Edgeworth,hall2013bootstrap,imbens2018causal}, we demonstrate that the causal bootstrap method offers a second-order refinement over the normal approximation. 
These contributions enhance the reliability of causal inference in complex experimental designs.

Our third contribution is the proposal of a novel causal bootstrap method based on constant-treatment-effect imputation, specifically designed for paired experiments.
Unlike \citet{imbens2018causal}'s causal bootstrap, which relies on rank-preserving imputation and requires each stratum to have at least two treated and two control units (a condition unmet in paired experiments), we impute unobserved potential outcomes by assuming a constant treatment effect for the bootstrap sample. 
We then generate a bootstrap distribution for pivotal statistics according to the actual experimental design and construct confidence intervals using the bootstrap quantile.
This approach is inspired by \citet{bind2020possible}, who emphasize the importance of examining the null randomization distribution for scientific insights, yet do not detail how to formally integrate it into statistical inference.
We validate the proposed causal bootstrap method based on constant-treatment-effect imputation and demonstrate its advantages through simulation studies and real data applications. Importantly, by utilizing a pivotal statistic, this causal bootstrap remains valid even when the constant treatment effect assumption does not hold for the true potential outcomes.

The paper proceeds as follows.
Section \ref{sec:sre} reviews design-based causal inference in stratified randomized experiments. Section \ref{sec:var} describes the sharp variance estimator and its asymptotic properties. In Section \ref{sec:bootrank}, we propose the causal bootstrap based on rank-preserving imputation and its theoretical properties in stratified randomized experiments. Section \ref{sec:bootsharp} provides the causal bootstrap method based on constant-treatment-effect imputation and its theoretical guarantees for paired experiments. Sections \ref{sec:sim} and \ref{sec:rd} present simulation studies and two real data applications, respectively. Finally, we conclude the paper in Section \ref{sec:dis}.

\section{Stratified Randomized Experiments}
\label{sec:sre}

We consider a design-based finite-population framework.
Following the Neyman--Rubin potential outcome model \citep{Rubin:1974,splawa1990application} and the stable unit treatment value assumption (SUTVA) \citep{Rubin1980} that for $i=1,\ldots,n$, each unit $i$ has two potential outcomes: $\yi[1]$ and $\yi[0]$, representing the outcomes if unit $i$ is allocated to the treatment or control group, respectively. 
Let $Z_i$ denote the treatment assignment, taking the value $1$ if unit $i$ is assigned to the treatment group and $0$ otherwise.
The observed outcome, denoted as $\yo$, can be expressed as $\yo = Z_i\yi[1] + (1-Z_i)\yi[0]$.
In the design-based finite-population inference framework,
randomness arises solely from the treatment assignment process and both $\yi[1]$ and $\yi[0]$ are considered fixed. 
Our goal is to infer the average treatment effect (ATE) $\tau=n^{-1}\sum_{i=1}^n\tau_i$, where $\tau_i=\yi[1]-\yi[0]$ represents the unit-level treatment effect. 

In stratified randomized experiments, the population is divided into $M$ strata based on pre-treatment covariates, such as gender. 
Each stratum consists of $\nm$ units, with the total population size given by $\sum_{m=1}^M\nm=n$. 
Within each stratum, complete randomization is performed, assigning $\nm[1]$ units to the treatment group and $\nm[0] = \nm - \nm[1]$ units to the control group. 
The probability that the treatment assignment vector $Z=(Z_1,\cdots, Z_n)$ takes a specific value $(z_1,\cdots,z_n)$ is given by $\prod_{m=1}^M (\nm[1]! \nm[0]!/\nm!)$, where $z_i \in {0,1}$, and $\sum_{i\in [m]}z_i=\nm[1]$ for $m=1,\cdots,M$, with $i \in [m]$ denoting units in stratum $m$. 

The average treatment effect (ATE) can be expressed as $\tau=\sum_{m=1}^M\pim\taum$, where $\pim=\nm/n$ is the proportion of of the population in stratum $m$, and $\taum=\nm^{-1}\sum_{i\in[m]}\{\yi[1]-\yi[0]\}$ is the ATE in the $m$-th stratum. 
We estimate $\taum$ using the difference-in-means estimator $\taumhat=\ymbarhat[1]-\ymbarhat[0]={\nm[1]}^{-1}\sum_{i\in[m]}Z_i\yo-{\nm[0]}^{-1}\sum_{i\in[m]}(1-Z_i)\yo$. 
The overall estimator for $\tau$ is the weighted mean of the stratum-specific estimators: $\hat \tau=\sum_{m=1}^M\pim\taumhat$.

To incorporate the sharp variance estimator and causal bootstrap methods, we introduce additional notations for the distribution functions. 
Within stratum $m$, the finite-population distribution functions for $\yi[1]$ and $\yi[0]$ are denoted by $\Gm=\nm^{-1} \sum_{i\in[m]}I\{\yi[1]\leq y\}$ and $\Fm=\nm^{-1} \sum_{i\in[m]}I\{\yi[0]\leq y\}$, respectively, where $I(\cdot)$ denotes the indicator function. 
The overall marginal distribution functions for $\yi[1]$ and $\yi[0]$ are denoted by $G(y)=\sum_{m=1}^M \pim \Gm$ and $F(y)=\sum_{m=1}^M \pim \Fm$.

The stratum-level ATE $\taum$ can be formulated as a functional: $\taum = \W(G_{[m]},F_{[m]}) = E_{G_{[m]}}\{\yi[1]\} - E_{F_{[m]}}\{\yi[0]\}$. Similarly, the overall ATE is given by $\tau = \W(G,F)=E_{G}\{\yi[1]\} - E_{F}\{\yi[0]\}$.
The stratum-level difference-in-means estimator $\taumhat$ can be expressed as $\taumhat=W(\hat G_{[m]},\hat F_{[m]})$, where $\Gmhat=\nm[1]^{-1} \sum_{i\in[m]}Z_iI\{\yi[1]\leq y\}$ and $\Fmhat=\nm[0]^{-1}\sum_{i\in[m]}(1-Z_i)I\{\yi[0]\leq y\}$ are the empirical distribution functions of $\yi[1]$ and $\yi[0]$ in stratum $m$, respectively. 
Similarly, $\hat \tau=W(\hat G,\hat F)$, where $\hat G(y)$ and $\hat F(y)$ denote the weighted means of $\Gmhat$ and $\Fmhat$ (weighted by $\pim$), respectively.

The variance of $\sqrt n( \hat{\tau} - \tau) $ is given by 
\begin{align}
    \sigma^2=n\sum_{m=1}^M\pim^2 \left(\frac{\varm[1]}{\nm[1]}+\frac{\varm[0]}{\nm[0]}-\frac{\varmdif}{\nm} \right) \label{sigma2},
\end{align}
where $\varm[z]$ ($z=0,1$) and $\varmdif$ are the finite-population variances of $\yi[z]$ and $\tau_i$, respectively. Specifically, $\varm[z]= (\nm-1)^{-1}\sum_{i\in [m]}\{\yi[z]-\ymbar\}^2$, where $\ymbar=\nm^{-1}\sum_{i\in[m]} \yi[z]$ for $z=0,1$. The term $\varmdif=(\nm - 1)^{-1}\sum_{i\in[m]}(\tau_i-\taum)^2$ captures the variation of individual treatment effects $\tau_i$'s within stratum $m$. By definition, $\varmdif=\varm[1]+\varm[0]-2\covm$, where 
$\covm=(\nm-1)^{-1}\sum_{i\in[m]}\{\yi[1]-\ymbar[1]\}\{\yi[0]-\ymbar[0]\}$.
Thus, the variance $\sigma^2$ can also be expressed as
\begin{align}
    \sigma^2&=\sum_{m=1}^M\pim\left(\frac{\nm[0]}{\nm[1]}\varm[1]+\frac{\nm[1]}{\nm[0]}\varm[0]+2\covm \right).
    \label{sigma2'}
\end{align}
We denote the joint distribution of potential outcomes in the $m$-th stratum as $\Hm=\nm^{-1}\sum_{i\in[m]}I\{\yi[1]\leq y_1,\yi[0]\leq y_0\}$ and in the whole population as $H(y_1,y_0)=n^{-1}\sum_{i=1}^nI\{\yi[1]\leq y_1,\yi[0]\leq y_0\}$. The variance $\sigma^2$ can be given by $\sigma^2=\sum_{m=1}^M\pim\V(G_{[m]},F_{[m]},H_{[m]})$, where $\V$ is a functional.

When $2\leq \nm[1]\leq \nm-2$, we estimate $\varm[z]$ by its sample analog $\varmhat[z]=(\nm[z]-1)^{-1}\sum_{i\in[m]}I(Z_i=z)\left(Y_i-\ymbarhat[z] \right)^2$.
However, estimating $\covm$ poses a challenge, since $\yi[1]$ and $\yi[0]$ cannot be observed simultaneously. Conservative estimators can be derived by estimating an upper bound for $\covm$.
\citet{splawa1990application} proposed using $\{\varm[1]+\varm[0]\}/2$ as an upper bound for $\covm$, leading to the Neyman-type variance estimator:
\begin{align*}
\sigmaneyman^2=n\sum_{m=1}^M\pim^2\left(\frac{\varmhat[1]}{\nm[1]}+\frac{\varmhat[0]}{\nm[0]} \right),
\end{align*}
where the subscript ``N'' stands for ``Neyman''. Under Conditions \ref{cond1}--\ref{cond3} below, $\hat \tau$ is asymptotically normal, and the Neyman-type variance estimator is conservative \citep{Liu2019}. Let $\rhom=\nm[1]/\nm$ denote the propensity score in stratum $m$, which is known by design.

\begin{condition}
    \label{cond1}
    There exist constants $\rhom^{\infty}$ and $C\in(0,0.5)$ such that $C<\min_{m=1,\cdots,M}\rhom^{\infty}\leq \max_{m=1,\cdots,M}\rhom^{\infty}<1-C$ and $\max_{m=1,\cdots,M}|\rhom- \rhom^{\infty}|\rightarrow 0$, as $n \rightarrow \infty$.
\end{condition}

\begin{condition}
    \label{cond2}
    As $n \rightarrow \infty$, the maximum stratum-specific squared distances divided by $n$ tend to zero, i.e.,  $n^{-1}\max_{m=1,\cdots,M}\max_{i\in[m]}\left\{\yi-\ymbar \right\}^2\rightarrow 0$ for $z = 0,1$. 
\end{condition}

\begin{condition}
    \label{cond3}
    The weighted variances 
$\sum_{m=1}^M\pim\varm[1]/\rhom$, $\ 
\sum_{m=1}^M\pim\varm[0]/(1-\rhom)$, and 
$\sum_{m=1}^M\pim\varmdif$, 
tend to finite limits, which are positive for the first two. The limit of $\sum_{m=1}^M\pim\{\varm[1]/\rhom+\varm[0]/(1-\rhom)-\varmdif\}$ is positive.
\end{condition}

Condition~\ref{cond1} stipulates that the stratum-specific propensity scores $\rhom$ are uniformly bounded away from zero and one. 
Condition~\ref{cond2} serves as a moment-type condition, which is less stringent than requiring a bounded $(2+\delta)$-th moment condition for some $\delta>0$. This condition is important for establishing finite-population asymptotic theory, as noted in previous studies \citep{li2016,Liu2019,schochet2022design}.
Condition~\ref{cond3} provides a sufficient condition for $\sigma^2$ to have a finite and positive limit, ensuring that the variance remains well-defined and meaningful in the context of our analysis.

\begin{proposition}
    \label{CLT}
    If Conditions \ref{cond1}--\ref{cond3} hold, the standard statistic $\sqrt n(\hat \tau-\tau)/\sigma$ converges in distribution to a standard normal distribution, denoted as $\mathcal{N}(0,1)$. 
    Furthermore, if $2 \leq \nm[1] \leq \nm-2$ for $m=1,\dots,M$, the Neyman-type variance estimator $\sigmaneyman^2$ converges in probability to the limit of $\sigma^2 + \sum_{m=1}^M\pim \varmdif $.
\end{proposition}

Proposition \ref{CLT} indicates that the sampling distribution of $\hat \tau$ can be well-approximated by the normal distribution, and the asymptotic variance can be conservatively estimated by $\sigmaneyman^2$ \citep{Liu2019}. A two-sided $1 - \alpha$ confidence interval for $\tau$ is $(\hat \tau-\sigmaneyman z_{1-\alpha/2}/\sqrt n,\hat \tau+\sigmaneyman z_{1-\alpha/2}/\sqrt n)$, where $z_{1-\alpha/2}$ is the lower $1-\alpha/2$ quantile of the standard normal distribution. The asymptotic coverage rate of this confidence interval is no less than $1-\alpha$ and is equal to $1-\alpha$ if and only if the treatment effects are stratum-specific additive, which means $\tau_i = \taum$ for all $i \in [m]$, or equivalently, $\varmdif = 0$.

\section{Sharp Variance Estimator}
\label{sec:var}

The Neyman-type variance estimator estimates the third term in \eqref{sigma2'} with an upper bound based on Cauchy-Schwarz inequality. To enhance this estimator, we derive a sharper upper bound for the third term and propose a consistent estimator for it.
Within stratum $m$, the sharp upper bound for $\covm$ is given by \citet{aronow2014}:
\begin{align*}
\covmh =\frac{\nm}{\nm-1}\left\{ \int_0^1 \Gminv \Fminv \mathrm du - \ymbar[1]\ymbar[0]\right\},
\end{align*}
where ``U'' indicates ``upper'', and $\Gminv$ and $\Fminv$ are the left-continuous inverse functions of the finite-population distribution functions $\Gm$ and $\Fm$, respectively.
By aggregating across all strata, we obtain the sharp upper bound for $\sigma^2$:
\begin{align*}
\sigma_S^2 = \sum_{m=1}^M\pim\left(\frac{\nm[0]}{\nm[1]}\varm[1] + \frac{\nm[1]}{\nm[0]}\varm[0] + 2\covmh \right),
\end{align*}
where ``S'' stands for ``sharp''. In the following part, we first provide a definition of co-monotonicity, then discuss the sharpness of $\sigma_S^2$. Finally, we demonstrate the consistency of our proposed estimator $\sigmaAGL^2$.

\begin{definition}[Co-monotonicity]
    Two subsets of $\mathbb R$ $ \{\yi[1]\}_{i=1}^n $ and $ \{\yi[0]\}_{i=1}^n $ are said to be co-monotonic if, for any pair of indices $ i $ and $ j $, $\{\yi[1]-\yj[1]\}\{\yi[0]-\yj[0]\}\geq 0$.
\end{definition}

The concept of co-monotonicity is widely used in many fields, including economics, finance, risk management and statistics \citep{dhaene2002concept,aronow2014,mcneil2015quantitative}. Based on this definition, we state the sharpness of $\sigma_S^2$ in the following proposition.

\begin{proposition}
    \label{prop:agl}
    Given $\{G_{[m]}(y)\}$ and $\{F_{[m]}(y)\}$ with no additional information about the joint distribution of $\{\yi[1], \yi[0]\}$, the bound $\sigma^2\leq \sigma_S^2$ is sharp. Equality holds if and only if $\yi[1]$ and $\yi[0]$ are co-monotonic in each stratum.
\end{proposition}

Denoting $\HmH=\min\{G_{[m]}(y_1),F_{[m]}(y_0)\}$ and $\HH=\sum_{m=1}^M\pim\HmH$ \citep{aronow2014}, the sharp variance can be expressed as $\sigma_S^2=\sum_{m=1}^M\pim\V(G_{[m]},F_{[m]},H_{[m]}^U)$.
By replacing population-level quantities with sample counterparts, we obtain a consistent estimator for the sharp upper bound:
\begin{align*}
    \sigmaAGL^2=\sum_{m=1}^M\pim\left(\frac{\nm[0]}{\nm[1]}\varmhat[1]+\frac{\nm[1]}{\nm[0]}\varmhat[0]+2\covmhhat \right),
\end{align*}
where
\begin{align*}
    \covmhhat=\frac{\nm}{\nm-1}\left\{\int_0^1\Gminvhat\Fminvhat\mathrm du-\ymbarhat[1]\ymbarhat[0]\right\},
\end{align*}
with $\Gminvhat$ and $\Fminvhat$ being the left-continuous inverse functions of $\Gmhat$ and $\Fmhat$, respectively.
The quantity $\sigmaAGL^2$ can be similarly expressed as $\sigmaAGL^2=\sum_{m=1}^M\pim \V(\hat G_{[m]},\hat F_{[m]},\hat H_{[m]}^U)$, where $\HmHhat=\min\{\Gmhat[y_1],\Fmhat[y_0]\}$.
The consistency of this estimator is demonstrated in the following theorem based on an additional condition.

\begin{condition}
    \label{cond4}
    (i) There exist constants $0< \delta < 1$ and $K>0$ independent of $n$, such that $\sup_n\{n^{-1}\sum_{i=1}^n|Y_i(z)|^{4+\delta}\}\leq K$, for $z=0,1$. (ii) As $n\rightarrow \infty$, $\sigma_S^2$ tends to a positive finite limit.
\end{condition}

Condition~\ref{cond4}(i) assumes a bounded $(4+\delta)$-th moment for the potential outcomes, which is crucial for applying a Glivenko--Cantelli type theorem in stratified randomized experiments under a general asymptotic regime.  A similar requirement is discussed by \citet{motoyama2023extended}.
Condition~\ref{cond4}(ii) ensures the regularity of the limit of the sharp bound \citep{imbens2018causal}.

\begin{theorem}[Consistency of sharp variance estimator]
\label{theoremAGLconvergence}
Under Conditions \ref{cond1}, \ref{cond3}, and \ref{cond4}, if $2 \leq \nm[1] \leq \nm-2$ for $m=1,\dots,M$,  then $\sigmaAGL^2 - \sigma_S^2 = o_P(1)$ and $\sqrt n(\hat \tau - \tau)/\sigmaAGL \stackrel{d}{\rightarrow} \mathcal{N}(0,\sigma^2/\sigma_S^2)$.
\end{theorem}

Theorem \ref{theoremAGLconvergence} supports different asymptotic regimes for strata, including (i) a fixed number of strata with increasing sizes, and (ii) a growing number of strata with bounded sizes.
In the first regime, extending the theoretical results of \citet{aronow2014} is straightforward.
However, under the second regime, bounding the tail probability of the empirical distribution becomes challenging, especially when combining the implicit upper bound across strata as $M\rightarrow \infty$.
To address these challenges, we employ Bobkov's inequality \citep{bobkov2004}, which helps establish a concentration inequality for the empirical distribution functions.
This approach allows us to determine convergence rates and provide an explicit upper bound on the tail probability in a Glivenko--Cantelli type theorem. Building on this, we derive the convergence rate of the sharp variance estimator, culminating in Theorem \ref{theoremAGLconvergence}.

Combining Proposition~\ref{prop:agl} and Theorem~\ref{theoremAGLconvergence}, we can construct an asymptotically narrowest valid $1 - \alpha$ confidence interval for $\tau$: $(\hat \tau-\sigmaAGL z_{1-\alpha/2}/\sqrt n,\hat \tau+\sigmaAGL z_{1-\alpha/2}/\sqrt n)$. The asymptotic coverage rate of this confidence interval is no less than $1- \alpha$ and is equal to $1- \alpha$ if and only if $\yi[1]$ and $\yi[0]$ are co-monotonic in each stratum. 

Although the sharp variance estimator produces an asymptotically narrowest confidence interval, combining it with the normal approximation may lead to anti-conservative results in finite-sample simulations, especially when the potential outcomes are co-monotonic \citep{aronow2014,imbens2018causal}. To address this problem, we propose a causal bootstrap method that achieves second-order refinement and offers better finite-sample performance.

\section{Causal Bootstrap Based on Rank-preserving Imputation}
\label{sec:bootrank}

The standard bootstrap, based on resampling with replacement, is widely used in sampling-based inference frameworks, where data is assumed to be independently and identically distributed (i.i.d.) from a super-population \citep{abadie2020sampling}. However, in contrast, we adopt a design-based finite-population causal inference framework, where randomness arises solely from the treatment assignment process rather than from the population itself.
To mimic this design-based randomness, the causal bootstrap treats potential outcomes as fixed and resamples treatment assignments according to the original experimental design. This approach is conceptually similar to Fisher randomization tests, but with a crucial distinction: the causal bootstrap does not rely on imputing unobserved potential outcomes based on a hypothesis. Instead, it seeks an imputation method such that the resulting bootstrap distribution closely approximates the true distribution of the estimator, offering a more accurate reflection of the design-based uncertainty.

To achieve this goal, we aim to find an imputation method that ensures the variance of the average treatment effect (ATE) estimator is greater than or equal to the upper bound of the true variance.
Proposition \ref{prop:agl} demonstrates that when $Y_i(1)$ and $Y_i(0)$ are co-monotonic within each stratum, the variance of the weighted difference-in-means estimator reaches the sharp upper bound. Thus, to generate a valid causal bootstrap, we can apply rank-preserving imputation within each stratum, ensuring that the imputed potential outcomes are co-monotonic. Specifically, in a stratified randomized experiment,  the causal bootstrap process based on rank-preserving imputation can be outlined as follows:

\textbf{1. Rank-preserving imputation}: 
For each stratum $m$, and for $i\in [m]$, the imputation of potential outcomes is as follows:
\begin{align*}
&\yimp[0]= \begin{cases}\yo & \text { if } Z_i=0 \\
\hat F_{[m]}^{-1}\left\{\Gmhat[\yo]\right\} & \text { if } Z_i=1 \end{cases},
\quad
\yimp[1]= \begin{cases}\hat G_{[m]}^{-1}\left\{\Fmhat[\yo]\right\} & \text { if } Z_i=0 \\
\yo & \text { if } Z_i=1 
\end{cases}.
\end{align*} 
This step ensures the imputed potential outcomes are co-monotonic, maintaining the rank-preserving property. The bootstrap version of the average treatment effect (ATE) is then calculated as $\taub=n^{-1}\sum_{i=1}^n\{\yimp[1]-\yimp[0]\}$.

\textbf{2. Assignments resampling}: Use actual stratified randomization to generate treatment assignments $Z_i^{*}$, $*=1,\cdots,B$, where $B$ is the number of bootstrap replications. Obtain the bootstrap version of the observed outcome $\yob = Z_i^*\yimp[1]+(1-Z_i^*)\yimp[0]$, $i=1,\cdots,n$.

\textbf{3. Calculate the pivotal statistics}: Calculate the bootstrap version ATE estimator $\taubhat$, the sharp variance estimator $(\sigmaAGL^*)^2$, and the $t$-statistics $T^*=\sqrt n(\taubhat-\taub)/\sigmaAGL^*$, $*=1,\cdots,B$, with $\{Y_i\}_{i=1}^{n}$ replaced by $\{\yob\}_{i=1}^{n}$.

\textbf{4. Construct the confidence interval}: Construct a two-sided $1- \alpha$ confidence interval for $\tau$: $(\hat \tau-\sigmaAGL q_{1-\alpha/2}^*/\sqrt n,\hat \tau-\sigmaAGL q_{\alpha/2}^*/\sqrt n)$, where $q_{1-\alpha/2}^*$ and $q_{\alpha/2}^*$ are the upper and lower $\alpha/2$ quantiles of the empirical distribution of $\{T^*\}_{*=1}^{B}$.

\begin{remark}\label{remark:equivalence}
In Step 1, the imputation could alternatively be implemented by copying the observed outcomes and sorting them accordingly; see the Supplementary Material for further details.
\end{remark}

The following theorem establishes the bootstrap distribution of $\sqrt n(\taubhat-\taub)/\sigmaAGL^*$, given $\{(Y_i,Z_i)\}_{i=1}^{n}$.

\begin{theorem}[Bootstrap CLT]
    \label{bootstrapCLT}
    Suppose that Conditions \ref{cond1}, \ref{cond3} and \ref{cond4} hold. If $2 \leq \nm[1] \leq \nm-2$ for $m=1,\dots,M$, then $\sqrt n(\taubhat-\taub)/\sigmaAGL^* \mid \{(Y_i,Z_i)\}_{i=1}^{n} \asyequi \mathcal{N}(0,1)$ holds in probability.
\end{theorem}

Referring back to Theorem~\ref{theoremAGLconvergence}, we note that $\sqrt n(\hat \tau - \tau)/\sigmaAGL$ converges to $\mathcal{N}(0,\sigma^2/\sigma_S^2)$ with $\sigma^2/\sigma_S^2 \leq 1$, implying that the bootstrap distribution of $\sqrt n(\taubhat-\taub)/\sigmaAGL^*$ provides a conservative approximation for the asymptotic distribution of $\sqrt n(\hat \tau - \tau)/\sigmaAGL$.
To establish a second-order refinement of the causal bootstrap method over the normal approximation, we derive the Edgeworth expansion for the scaled studentized statistic $T=(\sigma_S/\sigma)\sqrt n(\hat \tau-\tau)/\sigmaAGL$, where the scaling ensures $T$ is asymptotically standard normal based on \Cref{theoremAGLconvergence}, facilitating comparison with its bootstrap counterpart, $T^*$, in second-order asymptotics.

In the bootstrap method, we construct an interval for $T$, and leveraging the inequality $\sigma^2\leq \sigma_S^2$, we obtain a conservative interval for $\sqrt n(\hat \tau-\tau)/\sigmaAGL$. Consequently, the asymptotic coverage rate of the bootstrap confidence interval $(\hat \tau-\sigmaAGL q_{1-\alpha/2}^*/\sqrt n,\hat \tau-\sigmaAGL q_{\alpha/2}^*/\sqrt n)$ is at least $1-\alpha$, achieving equality if and only if $\yi[1]$ and $\yi[0]$ are co-monotonic within each stratum.

\begin{condition}
\label{cond5}
(i) There exists a constant $C>0$ independent of $n$, such that $\sup_n|Y_i(z)|\leq C$, $z=0,1$. 
(ii) The functions $F(y)$ and $G(y)$ converge to two non-lattice distribution functions $F^{\infty}(y)$ and $G^{\infty}(y)$ as $n\rightarrow \infty$, respectively. A distribution function $F^{\infty}(y)$ is non-latticed if $|\int\exp(\i ty)\mathrm dF^{\infty}(y)|\neq 1$ for all $t\neq 0$, where $\i^2=-1$.
(iii) For all $m=1,\ldots,M$, as $n\rightarrow\infty$, $\log \nm/\log n$ has a finite limit.
\end{condition}

Condition~\ref{cond5}(i) assumes that potential outcomes are bounded. The non-lattice assumption specified in Condition~\ref{cond5}(ii) is commonly employed in second-order refinement \citep{imbens2018causal,wang2022bootstrap}.
An intuitive sufficient condition for a distribution to be non-latticed is that the distribution is continuous or does not concentrate its probability mass on discrete, regularly spaced points.
Examples of non-lattice distributions include the Normal, Exponential, Gamma, and Uniform distributions.
Condition~\ref{cond5}(iii) is a technical assumption typically met in many practical scenarios, for instance, when either $\nm$ is bounded with $M\rightarrow \infty$, or $M$ is bounded with $\nm \rightarrow \infty$.  

\begin{theorem}[Edgeworth expansion]
\label{theorem.Edgeworth}
Under Conditions \ref{cond1}, \ref{cond3}, \ref{cond4}(ii), and \ref{cond5}, if $2 \leq \nm[1] \leq \nm-2$ for $m=1,\dots,M$, then, as $n\rightarrow \infty$, $T$ has an Edgeworth expansion:
\begin{align*}
\mathrm{pr}(T\leq t)&=\Phi(t)+\frac {\sqrt n}2E\left(\frac{\hat \tau-\tau}{\sigma}\frac{\sigmaAGL^2-\sigma_S^2}{\sigma_S^2}\right)t^2\phi(t)-\frac{n^{3/2}E(\hat \tau-\tau)^3}{6\sigma^3}(t^2-1)\phi(t)+o\left(n^{-1/2}\right),
\end{align*}
where $\Phi(\cdot)$ and $\phi(\cdot)$ are the distribution function and the density function of a standard normal random variable, respectively.
\end{theorem}

In comparison to the Edgeworth expansion derived by \citet{imbens2018causal}, the main challenge here lies in adapting to a diverse range of asymptotic regimes for the strata sizes. To address this, we classify the strata into three distinct groups based on their sizes: large strata, where the size is comparable to $n$; small strata, with bounded sizes; and medium strata, which fall between these two extremes. We handle each group separately, deriving the corresponding results. Further details are provided in the Supplementary Material.
By leveraging the Edgeworth expansion, we obtain the second-order refinement of the causal bootstrap, demonstrating its improvement over the normal approximation.

\begin{condition}
    \label{cond6}
    (i) The functionals $\W(G,F)$ and $\V(G,F,H)$ are three times Fr\'echet-differentiable with two bounded derivatives. (ii) The difference between $\{H_{[m]}(y_1,y_0)\}$ and $\{H_{[m]}^U(y_1,y_0)\}$ satisfies $\sum_{m=1}^M\pim \sup_{y_1,y_0}|H_{[m]}(y_1,y_0)-H_{[m]}^U(y_1,y_0)|=o(n^{-1/2})$.
\end{condition}

Condition~\ref{cond6}(i) ensures sufficient smoothness for two functionals; a similar assumption was also made by  \citet{imbens2018causal}. It ensures that the functionals behave in a predictable manner as the sample size $n$ grows, which is key for deriving accurate asymptotic results.
Condition~\ref{cond6}(ii) establishes precise control over the difference between the subgroup-specific true joint distributions $ H_{[m]}(y_1,y_0) $ and their co-monotonic counterparts $ H_{[m]}^U(y_1,y_0) $.
Under this condition, the difference of $\sigma_S^2$ and $\sigma^2$ is at most of order $o(n^{-1/2})$.
Without Condition~\ref{cond6}(ii), Theorems \ref{theoremAGLconvergence} and \ref{bootstrapCLT} demonstrate that both the normal approximation and the causal bootstrap serve as conservative approximations of the true distribution of $T$. Conservativeness implies that while the inference remains valid, the distribution approximation exhibits a first-order bias that dominates second-order differences. Thus, Condition~\ref{cond6}(ii) is necessary to eliminate this first-order bias and facilitate asymptotic second-order comparison as outlined in Theorem \ref{theorem.refinements}. 
More broadly, ruling out first-order bias through assumptions is essential for investigating second-order refinement, given the inherent conservativeness of design-based inference \citep{imbens2018causal}. We relax the slightly stronger assumption proposed by \citet{imbens2018causal} that $H_{[m]}(y_1, y_0) = H_{[m]}^U(y_1, y_0)$, which assumes fully co-monotonic potential outcomes within each stratum.
In contrast, Condition~\ref{cond6}(ii) allows for some discrepancy between the true and co-monotonic distributions, but limits how large this discrepancy can be.

\begin{theorem}[Second-order refinement]
\label{theorem.refinements}
Under Conditions \ref{cond1}, \ref{cond3}, \ref{cond4}(ii), \ref{cond5}, and \ref{cond6}, if $2 \leq \nm[1] \leq \nm-2$ for $m=1,\dots,M$, then, as $n\rightarrow \infty$, $\sup_{t\in \mathbb R}|\mathrm{pr}(T\leq t)-\Phi(t)|=O(n^{-1/2})$, $\sup_{t\in \mathbb R}|\mathrm{pr}(T\leq t)-\mathrm{pr}^*(T^*\leq t)|=o_P(n^{-1/2})$,
where $\mathrm{pr}^*$ denotes the conditional probability induced by the bootstrap assignments $Z_i^*$ given the original data $\{(Y_i,Z_i)\}_{i=1}^{n}$. 
\end{theorem}

In contrast to the normal approximation, which has a convergence rate of $O(n^{-1/2})$, the causal bootstrap method approximates the true distribution function with a faster rate of $o_P (n^{-1/2})$.
This result underscores the advantages of employing the causal bootstrap, especially in scenarios with small or moderate sample sizes. These benefits are further illustrated through numerical results in Sections \ref{sec:sim} and \ref{sec:rd}.
The error rate of $o(n^{-1/2})$ arises from the Edgeworth expansion for the studentized statistic $T$ with an \textit{estimated variance} and is also reported in the literature on Edgeworth expansions for other studentized statistics in finite-population inference \citep{bloznelis2003edgeworth,wang2022bootstrap}. In contrast, \citet{imbens2018causal} achieved an explicit rate of $O(n^{-1})$ by employing the Edgeworth expansion for statistics with a \textit{known variance} \citep{bloznelis2002edgeworth}.

\section{Causal Bootstrap Based on Constant-treatment-effect Imputation}
\label{sec:bootsharp}

While paired experiments are special cases of stratified randomized experiments with $\nm[1] = \nm[0] = 1$, they present unique challenges. These challenges arise from their small strata size, causing methods like the sharp variance estimator and rank-preserving imputation to fail.
Table~\ref{toyexample} illustrates this with a toy example where $n = 4$ units are grouped into $M = 2$ pairs. 
In this example, rank-preserving imputation results in two identical units in each pair, which results in the difference-in-means estimator being fixed at 4, regardless of randomization. Consequently, the bootstrap distribution degenerates to a single point, failing to approximate the true distribution.
This phenomenon isn't limited to just this toy example; it occurs across all paired experiments. Therefore, rank-preserving imputation method breaks down, making it necessary to explore alternative imputation methods to establish a valid and feasible causal bootstrap procedure tailored for paired experiments.

\begin{table*}[!htbp]%
\caption{A toy example demonstrates that rank-preserving imputation fails for paired experiments.\label{toyexample}\hspace{0.5\textwidth}}
\begin{tabular*}{0.45\textwidth}{@{\extracolsep\fill}ccccc@{}}
\toprule
\multicolumn{5}{c}{Observed Outcomes} \\ 
\midrule
$m$ & $\yi[1]$ & $\yi[0]$ & $Z_i$ & $Y_i$\\ 
\midrule
1 & 8 & {\color{red} ?} & 1 & 8\\ 
1 & {\color{red} ?} & 4 & 0 & 4\\ 
2 & 6 & {\color{red} ?} & 1 & 6\\ 
2 & {\color{red} ?} & 2 & 0 & 2\\ 
\bottomrule
\end{tabular*}
\hfill
\begin{tabular*}{0.45\textwidth}{@{\extracolsep\fill}ccccc@{}}
  \toprule
\multicolumn{5}{c}{Rank-preserving Imputed Outcomes} \\ \midrule
$m$ & $\yi[1]$ & $\yi[0]$ & $Z_i$ & $Y_i$\\ \midrule
1 & 8 & {\color{red} (4)} & 1 & 8\\ 
1 & {\color{red} (8)} & 4 & 0 & 4\\ 
2 & 6 & {\color{red} (2)} & 1 & 6\\ 
2 & {\color{red} (6)} & 2 & 0 & 2\\ 
\bottomrule
\end{tabular*}
\end{table*}

Recalling that if we identify an imputation under which the variance of the ATE estimator is greater than or equal to the upper bound of the true variance, intuitively, the causal bootstrap based on this imputation should be valid, as shown in Section~\ref{sec:bootrank}.
The Neyman-type variance estimator is conservative, as its asymptotic limit serves as an upper bound for the true variance. 
This upper bound is achieved when constant treatment effect assumption holds, i.e., $Y_i(1)=Y_i(0)+\Delta$ for all $i$, with a constant $\Delta$. 
Motivated by this, we propose imputing unobserved potential outcomes based on the constant treatment effect assumption.
Moreover, we will prove the validity of the corresponding causal bootstrap without assuming a constant treatment effect.

We now briefly review some known results for design-based causal inference in paired experiments before introducing the causal bootstrap procedure.
In paired experiments, there are $M = n/2$ pairs, with each pair consisting of two units ($n_{[m]} = 2$). Let the potential outcomes for the two units in pair $m$ be denoted as $\yma$ and $\ymb$, respectively. The average treatment effect (ATE) is given by $\tau = \sum_{m=1}^M \pim \taum = M^{-1} \sum_{m=1}^M \taum$, where
$\taum = \{\yma[1] + \ymb[1]\}/2 - \{\yma[0] + \ymb[0]\}/2$.
For unit $i$ in the $m$-th pair, $Z_{[m]i}$ represents the treatment assignment, and $Y_{[m]i}$ denotes the observed outcome, defined as $Y_{[m]i} = Z_{[m]i}Y_{[m]i}(1) + (1 - Z_{[m]i})Y_{[m]i}(0)$ for $i=1,2$.
The weighted difference-in-means estimator $\hat{\tau}=M^{-1}\sum_{m=1}^M\taumhat$, where $\taumhat=\zm\{\ymao-\ymbo\}+(1-\zm)\{\ymbo-\ymao\}$.
The variance of $\sqrt{n}(\hat{\tau} - \tau)$ is given by $\sigma^2=n\mathrm{var}(\hat \tau)=n^{-1}\sum_{m=1}^M\left\{\yma[1]+\yma[0]-\ymb[1]-\ymb[0] \right\}^2$.
A commonly used Neyman-type variance estimator is $\hat \sigma_{\p}^2=2\sum_{m=1}^M(\taumhat -\hat\tau)^2/(M-1)$, which is generally conservative and is consistent if the treatment effects $\taum$ are constant across pairs \citep{imai2008pair}. Its expectation is $\sigma_{\p}^2=E(\hat \sigma_{\p}^2)=\sigma^2+2\sum_{m=1}^M(\taum-\tau)^2/(M-1)$, which is generally not sharp when the marginal distributions $\Gm$ and $\Fm$ are estimable, as in \Cref{prop:agl}. However, $\Gm$ and $\Fm$ cannot be estimated from a single observation in each pair, motivating us to consider an alternative estimable class to examine the sharp variance upper bound in paired experiments. \Cref{proposition.pair} establishes the sharpness of $\sigma_{\p}^2$ within the class of unbiasedly estimable variance upper bounds. This notion underscores that the sharpness of a variance upper bound should be defined within an estimable class specific to the experimental setting \citep{harshaw2021optimized}.
\begin{proposition}
    \label{proposition.pair}
    The bound $\sigma^2\leq \sigma_{\p}^2$ is sharp within the class of unbiasedly estimable variance upper bounds of $\sigma^2$ in paired randomized experiments. Equality holds if and only if the treatment effects $\taum$ are constant across pairs.
\end{proposition}
Leveraging asymptotic normality, we construct a conservative Wald-type confidence interval $(\hat \tau-z_{1-\alpha/2}\hat\sigma_{\p}/\sqrt n,\hat \tau+z_{1-\alpha/2}\hat\sigma_{\p}/\sqrt n)$. 
To improve the finite-sample performance, we propose a causal bootstrap based on the
constant-treatment-effect imputation as follows:

\textbf{1. Constant-treatment-effect imputation}: 
For each pair $m=1,\ldots,M$ and unit $i=1,2$, impute the unobserved potential outcomes as follows:
\begin{align*}
&\ymiimp[0]= \begin{cases}Y_{[m]i} & \text { if } Z_{[m]i}=0 \\
Y_{[m]i} -\Delta & \text { if } Z_{[m]i}=1 \end{cases},
\quad
\ymiimp[1]= \begin{cases}Y_{[m]i} +\Delta & \text { if } Z_{[m]i}=0 \\
Y_{[m]i} & \text { if } Z_{[m]i}=1 
\end{cases},
\end{align*} 
where $\Delta=\hat \tau$. The bootstrap version of the ATE is $\taub=n^{-1}\sum_{m=1}^M\sum_{i=1}^2\{\ymiimp[1]-\ymiimp[0]\}$.

\textbf{2. Assignments resampling}: Using the actual paired randomization procedure, generate treatment assignments $Z_{[m]i}^*$, for $*=1,\cdots,B$. Obtain the bootstrap version of the observed outcome $Y_{[m]i}^*=Z_{[m]i}^*\ymiimp[1] + (1-Z_{[m]i}^*)\ymiimp[0]$.

\textbf{3. Calculate the pivotal statistics}: 
    Calculate the bootstrap version ATE estimator $\taubhat$, the variance estimator $(\hat \sigma_{\p}^*)^2$, an thed $t$-statistics $T^*_{\p}=\sqrt n(\taubhat-\taub)/\hat \sigma_{\p}^*$, $*=1,\cdots,B$, with $\{\ymao,\ymbo\}_{m=1}^{M}$ replaced by $\{\ymao^*,\ymbo^*\}_{m=1}^{M}$.

\textbf{4. Construct the confidence interval}: Construct a two-sided $1- \alpha$ confidence interval for $\tau$: $(\hat \tau-\hat \sigma_{\p}q_{1-\alpha/2}^*/\sqrt n,\hat \tau-\hat \sigma_{\p}q_{\alpha/2}^*/\sqrt n)$, where $q_{1-\alpha/2}^*$ and $q_{\alpha/2}^*$ are the upper and lower $\alpha/2$ quantiles of the empirical distribution of $\{T^*_{\p}\}_{*=1}^{B}$.

The following theorem confirms that the causal bootstrap based on constant-treatment-effect imputation could serve as an approximation for the true distribution of $\sqrt n(\hat \tau - \tau)/\hat \sigma_{\p}$.

\begin{condition}
    \label{cond7}
    (i) There exists $L>0$ independent of $n$, such that $\sup_n[n^{-1}\sum_{i=1}^n\{\yi[z]\}^4]\leq L$ for $z=0,1$.
    (ii) As $n\rightarrow \infty$, $\sigma_{\p}^2$ tends to a positive finite limit.
\end{condition}

Condition \ref{cond7}(i) assume bounded fourth moments for the potential outcomes, a condition also used by \citet{wu2021} for stratified randomized experiments.
Condition~\ref{cond7}(ii) is analogous to Condition~\ref{cond4}(ii), ensuring the regularity of the bootstrap variance.

\begin{theorem}[Bootstrap CLT]
    \label{theorem.pair.bootstrapCLT}
    Under Conditions~\ref{cond1}, \ref{cond3} and \ref{cond7}, we have $\sqrt n(\taubhat-\taub)/\hat \sigma_{\p}^*\mid \{(\ymao,\ymbo,Z_{[m]1},Z_{[m]2})\}_{m=1}^M\stackrel d{\rightarrow} \mathcal N(0,1)$ holds in probability.
\end{theorem}

Theorem~\ref{theorem.pair.bootstrapCLT} establishes that under the specified conditions, the distribution of the bootstrapped statistic converges to a standard normal distribution in probability. This implies that the causal bootstrap procedure is valid for constructing confidence intervals in paired experiments. Importantly, this result holds even if the constant treatment effect assumption used in the imputation procedure does not hold in reality, making the approach robust.

The central idea behind the causal bootstrap's validity lies in the use of pivotal or studentized statistics $T^*_{\p}=\sqrt n(\taubhat-\taub)/\hat \sigma_{\p}^*$, which is crucial in design-based finite-population causal inference \citep{cohen2022gaussian}.
\citet{ding2017} showed that Fisher randomization tests can be anti-conservative for testing the weak null hypothesis $H_0:\tau=0$ when using sharp null hypotheses $\yi[1]=\yi[0]$. 
However, as \citet{wu2021} demonstrated, the use of studentized statistics corrects this issue, ensuring asymptotically valid inferences for testing the weak null.

We examine the second-order asymptotic properties of both the normal approximation and the causal bootstrap for paired experiments in the Supplementary Material. Similar to Section \ref{sec:bootrank}, the causal bootstrap based on constant-treatment-effect imputation achieves a convergence rate of $o_P(n^{-1/2})$. However, unlike Section \ref{sec:bootrank}, the normal approximation also achieves a convergence rate of $o(n^{-1/2})$ rather than $O(n^{-1/2})$ in Theorem \ref{theorem.refinements}. This improvement arises from the favorable properties of the assignment variable $\zm$ in paired experiments, which follows an i.i.d. Bernoulli distribution and enhances the normal approximation's accuracy. Consequently, the second-order refinement of the causal bootstrap vanishes. Nevertheless, numerical results for paired experiments in Sections \ref{sec:sim} and \ref{sec:rd} demonstrate that the causal bootstrap outperforms the normal approximation when the potential outcomes' distribution is heavy-tailed or deviates significantly from normality. 
Theoretical comparison of causal bootstrap and normal approximation in those cases is left for future work.

\section{Simulation}
\label{sec:sim}

\subsection{Stratified Randomized Experiments}

In this section, we assess the finite-sample performance of the proposed methods in stratified randomized experiments.
We begin by examining the sharp variance estimator, followed by a comparison of three inference approaches: the normal approximation with the Neyman-type variance estimator (Neyman+Normal), the normal approximation with the sharp variance estimator (Sharp+Normal), and the causal bootstrap based on rank-preserving imputation with the sharp variance estimator (Sharp+Bootstrap).

Our data-generating process (DGP) involves three key components: stratum configurations, the proportion of treated units, and the relationship between $\yi[0]$ and $\yi[1]$.
For the stratum configurations, we consider $M=10$ or $M=20$ strata, with stratum sizes $\nm=10$ or $\nm=40$. For the proportion of treated units (i.e., propensity scores), we consider scenarios involving equal or unequal propensity scores across strata. In scenarios with equal propensity scores, we set $\rhom=0.5$ for all strata. In scenarios with unequal propensity scores, we set $\rhom=0.4$ for the first half of the strata and $\rhom=0.6$ for the second half. For the potential outcome distributions, we consider four cases:

\textbf{Case 1 (Additive)}: $\yi[1]$ is generated from $\text{Gamma}(1,1)$ independently and $\yi[0]=\yi[1]-\tau$ for $i=1,\cdots,n$. 
In this case, the treatment effects are additive, and $\yi[1]$ and $\yi[0]$ are co-monotonic, ensuring $\sigmaneyman^2$ and $\sigmaAGL^2$ are consistent. We present results with $\tau=0$ in the main text and $\tau=1$ in the Supplementary Material.
    
\textbf{Case 2 (Co-monotonic)}: Both $\yi[1]$ and $\yi[0]$ are generated from $\text{Gamma}(1,1)$ independently, followed by sorting them respectively to ensure $\yi[1]$ and $\yi[0]$ are co-monotonic. 
In this case, $\sigmaneyman^2$ is conservative, whereas $\sigmaAGL^2$ is consistent.

\textbf{Case 3 (Dependent)}: $\yi[1]$ is generated from $\text{Gamma}(1,1)$ independently and $\yi[0]=\yi[1]+\varepsilon_i$, where $\varepsilon_i\sim \mathcal N(0,0.5^2)$. In this case, both $\sigmaneyman^2$ and $\sigmaAGL^2$ are conservative.
    
\textbf{Case 4 (Independent)}: Both $\yi[1]$ and $\yi[0]$ are generated from $\text{Gamma}(1,1)$ independently. In this case, both $\sigmaneyman^2$ and $\sigmaAGL^2$ are conservative.

Since we consider the design-based inference, potential outcomes are generated once and kept fixed. The randomness arises solely from the treatment assignments, which are generated for 1000 replications.

\begin{table*}[!htbp]%
\caption{Square root of the ratio between the sharp variance estimator and the true variance.\label{table.sharp}\hspace{0.8\textwidth}}
\begin{tabular*}{\textwidth}{@{\extracolsep\fill}ccccccc@{}}
\toprule
\multicolumn{2}{c}{} & Case & Case 1 & Case 2 & Case 3 & Case 4\\
\cmidrule{3-7}
$\rhom$ & $M$ & $\nm$ & $\sigmaAGL/\sigma$ & $\sigmaAGL/\sigma$ & $\sigmaAGL/\sigma$ & $\sigmaAGL/\sigma$ \\
\midrule
\multirow{4}*{Equal} & 10 & 10 & 0.89 & 0.92 & 0.91 & 1.17 \\ 
& 10 & 40 & 0.96 & 0.96 & 0.98 & 1.37 \\ 
& 20 & 10 & 0.89 & 0.92 & 0.90 & 1.17 \\ 
& 20 & 40 & 0.96 & 0.96 & 0.98 & 1.36 \\ 
\midrule
\multirow{4}*{Unequal} & 10 & 10 & 0.88 & 0.91 & 0.89 & 1.14 \\ 
& 10 & 40 & 0.96 & 0.96 & 0.97 & 1.35 \\ 
& 20 & 10 & 0.88 & 0.92 & 0.89 & 1.14 \\ 
& 20 & 40 & 0.95 & 0.96 & 0.98 & 1.33 \\ 
\bottomrule
\end{tabular*}
\begin{tablenotes}
\item Case 1: Additive; Case 2: Co-monotonic; Case 3: Dependent; Case 4: Independent.
\end{tablenotes}
\end{table*}

Table~\ref{table.sharp} presents the square root of the ratio between the sharp variance estimator $\sigmaAGL$ and the true variance $\sigma$.
In cases 1 and 2, where potential outcomes are co-monotonic, the asymptotic theory suggests that the sharp variance estimator $\sigmaAGL^2$ should be a consistent estimator for $\sigma_S^2$. However, the simulation results show that $\sigmaAGL^2$ tends to underestimate $\sigma^2$, which possibly arises from the non-linearity of the functional $V$ \citep{imbens2018causal}.
In Cases 3 and 4, the asymptotic theory suggests that the sharp variance estimator $\sigmaAGL^2$ should be conservative. While the estimator is indeed conservative in Case 4, it tends to underestimate $\sigma^2$ in Case 3, particularly in smaller sample sizes. This discrepancy highlights that despite the theoretical expectation of conservativeness, practical simulations reveal scenarios where the estimator may fall short.
This underestimation of variance can lead to under-coverage of confidence intervals derived from normal approximations. Fortunately, this issue can be mitigated by combining the sharp variance estimator with the causal bootstrap, enhancing the reliability of inference.

\begin{table*}[!htbp]%
\caption{Simulation results for equal propensity scores.\label{tab:sim_eq}\hspace{0.8\textwidth}}
\begin{tabular*}{\textwidth}{@{\extracolsep\fill}ccccccccccc@{}}
\toprule
\multicolumn{2}{c}{} & Case & \multicolumn{2}{c}{Case 1} & \multicolumn{2}{c}{Case 2} & \multicolumn{2}{c}{Case 3} & \multicolumn{2}{c}{Case 4}\\
\cmidrule{3-11}
$M$ & $\nm$ & Method & Cover & Length & Cover & Length & Cover & Length & Cover & Length\\
\midrule
\multirow{3}*{10} & \multirow{3}*{10} & N+N & 0.955 & 0.781 & 0.961 & 0.695 & 0.954 & 0.976 & 0.993 & 0.691 \\ 
& & S+N & 0.925 & 0.694 & 0.925 & 0.597 & 0.917 & 0.867 & 0.978 & 0.615 \\ 
& & S+B & 0.950 & 0.766 & 0.943 & 0.655 & 0.946 & 0.970 & 0.991 & 0.680 \\ 
\midrule 
\multirow{3}*{10} & \multirow{3}*{40} & N+N & 0.950 & 0.368 & 0.956 & 0.386 & 0.954 & 0.454 & 0.997 & 0.396 \\ 
& & S+N & 0.945 & 0.353 & 0.945 & 0.366 & 0.943 & 0.433 & 0.994 & 0.381 \\ 
& & S+B & 0.947 & 0.365 & 0.946 & 0.378 & 0.949 & 0.448 & 0.996 & 0.395 \\ 
\midrule  
\multirow{3}*{20} & \multirow{3}*{10} & N+N & 0.954 & 0.490 & 0.969 & 0.523 & 0.948 & 0.652 & 0.992 & 0.524 \\ 
& & S+N & 0.927 & 0.436 & 0.918 & 0.444 & 0.925 & 0.573 & 0.979 & 0.465 \\ 
& & S+B & 0.948 & 0.475 & 0.943 & 0.478 & 0.940 & 0.624 & 0.988 & 0.506 \\ 
\midrule
\multirow{3}*{20} & \multirow{3}*{40} & N+N & 0.957 & 0.281 & 0.964 & 0.278 & 0.957 & 0.305 & 0.995 & 0.283 \\ 
& & S+N & 0.950 & 0.269 & 0.953 & 0.262 & 0.948 & 0.291 & 0.995 & 0.271 \\ 
& & S+B & 0.956 & 0.277 & 0.957 & 0.271 & 0.950 & 0.300 & 0.995 & 0.280 \\ 
\bottomrule
\end{tabular*}
\begin{tablenotes}
\item Case 1: Additive; Case 2: Co-monotonic; Case 3: Dependent; Case 4: Independent.
\item N+N is short for Neyman+Normal; S+N is short for Sharp+Normal; S+B is short for Sharp+Bootstrap.
\item Cover: Empirical coverage probability of $95\%$ confidence interval. 
\item Length: Mean confidence interval length. 
\end{tablenotes}
\end{table*}

Table~\ref{tab:sim_eq} presents the simulation results regarding the empirical coverage probability and mean confidence interval length for 95\% confidence intervals in scenarios with equal propensity scores. Similar results for unequal propensity scores are available in the Supplementary Material.
In all cases, both the Neyman+Normal (N+N) and Sharp+Bootstrap (S+B) methods consistently achieve a nominal coverage of 95\%, with N+N displaying a more conservative approach than S+B. The Sharp+Normal (S+N) method generally exhibits under-coverage, except in Case 4, where it demonstrates conservative behavior, as shown in Table~\ref{table.sharp}.
Notably, S+B reduces the interval length by up to 8.7\% compared to the traditional N+N method, highlighting a significant efficiency gain. In large sample sizes, S+B and N+N show equivalent performance. In Case 4, S+N reduces the interval length by up to 11.4\% relative to N+N.

Overall, these simulation results support the theoretical advancements of the causal bootstrap. Therefore, in stratified randomized experiments, we recommend employing the causal bootstrap based on rank-preserving imputation, combined with the sharp variance estimator, for its robust and efficient performance.

\subsection{Paired Experiments}

In this section, we evaluate the finite-sample performance of the proposed methods in paired experiments. We compare two inference approaches: the normal approximation (Normal) and the causal bootstrap based on constant-treatment-effect imputation (Bootstrap).

Our data-generating process (DGP) involves three key components: sample size, the distribution of potential outcomes, and the relationship between $\yi[0]$ and $\yi[1]$.
We vary the number of pairs $M$ from 30 to 100. Potential outcomes $\yi[1]$ are generated from either a Gamma distribution $\text{Gamma}(1/10,10)$ or a Normal distribution $\mathcal N(0,1)$ independently. We also consider uniform and Pareto distributions, with those results included in the Supplementary Material.
We consider two cases of the relationship between $\yi[0]$ and $\yi[1]$:

\textbf{Case 1 (Additive)}: $\yi[0]=\yi[1]-\tau$ for $i=1,\cdots,n$. In this case, the estimated variance $\hat \sigma_{\p}^2$ is both unbiased and consistent. Results with $\tau=0$ are shown in the main text, while results with $\tau=1$ are similar and available in the Supplementary Material. 

\textbf{Case 2 (Independent)}: $\yi[0]$ is generated from the same distribution but independent of $\yi[1]$. In this case, the variance estimator $\hat \sigma_{\p}^2$ is conservative. 

\begin{table*}[!htbp]%
\caption{Simulation results in paired experiments.\label{table.pair}\hspace{\textwidth}}
\begin{tabular*}{\textwidth}{@{\extracolsep\fill}cccccccccc@{}}
\toprule
\multicolumn{2}{c}{} & \multicolumn{4}{c}{Gamma} & \multicolumn{4}{c}{Normal}\\
\cmidrule{3-6} \cmidrule{7-10}
& Case & \multicolumn{2}{c}{Case 1} & \multicolumn{2}{c}{Case 2} & \multicolumn{2}{c}{Case 1} & \multicolumn{2}{c}{Case 2}\\
\cmidrule{2-10}
$M$ & Method & Cover & Length & Cover & Length & Cover & Length & Cover & Length\\
\midrule
\multirow{2}*{30} & Neyman & 0.975 & 1.930 & 0.990 & 3.985 & 0.940 & 0.753 & 0.989 & 0.992 \\ 
& Bootstrap & 0.961 & 1.762 & 0.981 & 3.689 & 0.952 & 0.789 & 0.994 & 1.034 \\ 
\midrule
\multirow{2}*{40} & Neyman & 0.966 & 1.595 & 0.991 & 2.989 & 0.937 & 0.629 & 0.987 & 0.869 \\ 
& Bootstrap & 0.952 & 1.531 & 0.980 & 2.738 & 0.946 & 0.649 & 0.989 & 0.896 \\ 
\midrule
\multirow{2}*{50} & Neyman & 0.983 & 2.355 & 0.992 & 2.527 & 0.956 & 0.632 & 0.982 & 0.735 \\ 
& Bootstrap & 0.963 & 2.175 & 0.983 & 2.391 & 0.961 & 0.649 & 0.986 & 0.753 \\ 
\midrule
\multirow{2}*{60} & Neyman & 0.954 & 3.103 & 0.995 & 2.161 & 0.949 & 0.630 & 0.987 & 0.696 \\ 
& Bootstrap & 0.950 & 3.037 & 0.987 & 2.059 & 0.957 & 0.644 & 0.989 & 0.710 \\ 
\midrule
\multirow{2}*{70} & Neyman & 0.954 & 2.667 & 0.987 & 2.123 & 0.950 & 0.602 & 0.986 & 0.623 \\ 
& Bootstrap & 0.949 & 2.603 & 0.978 & 2.054 & 0.952 & 0.613 & 0.988 & 0.633 \\ 
\midrule
\multirow{2}*{80} & Neyman & 0.958 & 2.336 & 0.991 & 1.917 & 0.941 & 0.569 & 0.991 & 0.607 \\ 
& Bootstrap & 0.945 & 2.274 & 0.986 & 1.870 & 0.946 & 0.577 & 0.991 & 0.616 \\ 
\midrule
\multirow{2}*{90} & Neyman & 0.964 & 2.098 & 0.995 & 1.791 & 0.945 & 0.540 & 0.993 & 0.556 \\ 
& Bootstrap & 0.956 & 2.041 & 0.993 & 1.759 & 0.945 & 0.547 & 0.996 & 0.563 \\ 
\midrule
\multirow{2}*{100} & Neyman & 0.952 & 1.922 & 0.995 & 1.625 & 0.936 & 0.507 & 0.990 & 0.531 \\ 
& Bootstrap & 0.945 & 1.873 & 0.992 & 1.595 & 0.948 & 0.513 & 0.991 & 0.538 \\ 
\bottomrule
\end{tabular*}
\begin{tablenotes}
\item Case 1: Additive; Case 2: Independent.
\item Cover: Empirical coverage probability of $95\%$ confidence interval.
\item Length: Mean confidence interval length.
\end{tablenotes}
\end{table*}

Table~\ref{table.pair} indicates that both the normal approximation and the causal bootstrap consistently achieve the nominal coverage of 95\% across all cases. When potential outcomes are generated from the heavy-tailed Gamma distribution, the causal bootstrap reduces the interval length by 1.8\% to 8.9\% compared to the normal approximation. However, when outcomes are generated from the normal distribution, the causal bootstrap does not offer any advantage over the normal approximation. Similar results for the uniform distribution (which has no outliers) and the Pareto distribution (which is heavy-tailed) are available in the Supplementary Material.

Overall, in the presence of outliers or significant deviations from normality in the outcome distribution, we recommend using the causal bootstrap for enhanced efficiency compared to the normal approximation. In practical applications, density plots and quantile-quantile (Q-Q) plots can be utilized to identify outliers and assess the normality of observed outcomes before selecting the appropriate inference method. A real data analysis illustrating this approach is provided in the following section.

\section{Real Data Analysis}
\label{sec:rd}

\subsection{A Clinical Trial for Cannabis Cessation}

In this section, we apply the causal bootstrap method based on rank-preserving imputation in a stratified randomized experiment evaluating the effect of N-acetylcysteine (NAC) versus placebo (PBO) on cannabis cessation \citep{mcclure2014achieving}. 
Participants were stratified by site and self-reported tobacco smoking status, yielding $M=7$ strata. 
After excluding 38 individuals with missing outcomes and 2 with extreme outcomes, 214 participants remained, with 111 in the treatment group and 113 in the control group.
The primary outcome is the creatinine-normalized tetrahydrocannabinol level, which is the ratio of the cannabinoid level to the urine creatinine level and is typically regarded as a quantitative measure of marijuana metabolites \citep{huestis1998differentiating,schwilke2011differentiating}.
Figure~\ref{fig:clinical} shows that the distribution of this outcome is skewed.

\begin{figure*}[!htbp]
    \centerline{\includegraphics[width=\textwidth]{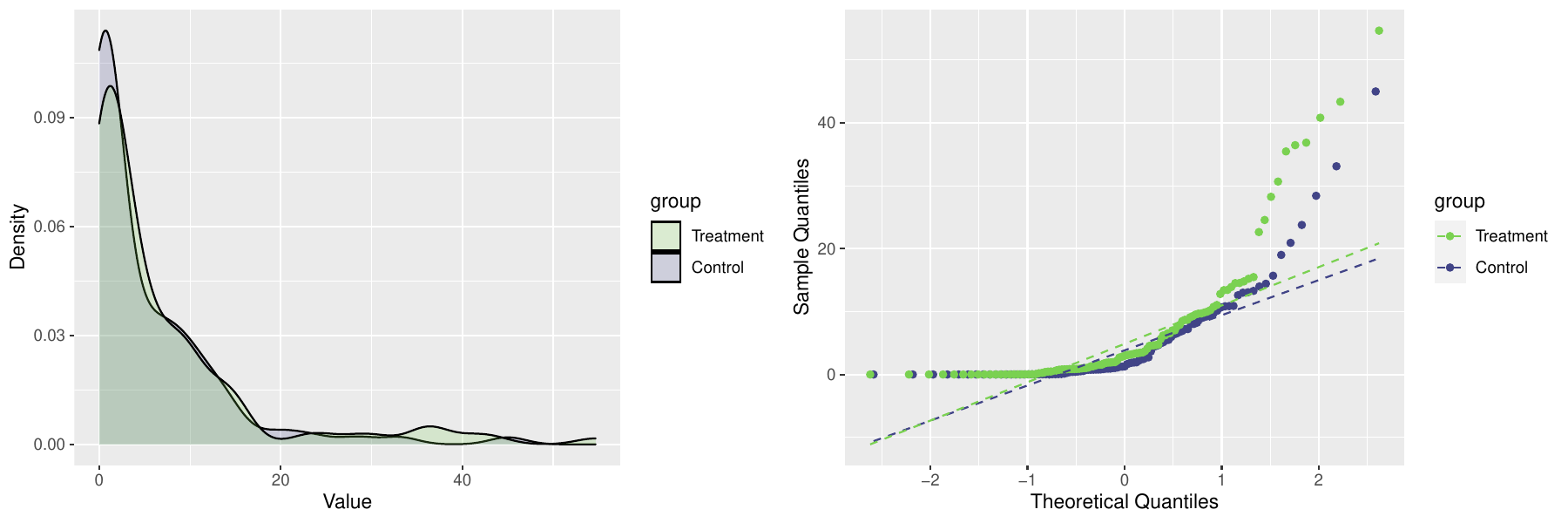}}
    \caption{Density plot and Q-Q plot for the outcomes from the clinical trial for cannabis cessation.\label{fig:clinical}\hspace{0.5\textwidth}}
\end{figure*}

We use three methods to construct 95\% confidence intervals for the ATE: the normal approximation with the Neyman-type variance estimator (Neyman+Normal), the normal approximation with the sharp variance estimator (Sharp+Normal), and the causal bootstrap based on rank-preserving imputation with the sharp variance estimator (Sharp+Bootstrap). The resulting intervals are $(-0.56, 4.27)$, $(-0.33, 4.04)$, and $(-0.30, 4.11)$, respectively, all encompassing 0, indicating insufficient evidence to support the assert NAC's effectiveness in enhancing cannabis cessation. Notably, Sharp+Normal and Sharp+Bootstrap reduced interval lengths by 9.4\% and 8.8\% compared to Neyman+Normal, respectively, demonstrating enhanced inference efficiency.

To evaluate the repeated sampling performance of the proposed methods, we generate semi-synthetic dataset imputing the counterfactual outcomes.
In stratum $m$, we use normal distributions $\mathcal N\{\bar Y_{[m]z},\varmhat[z]\}$ to impute the counterfactual potential outcomes under treatment and control, appling Neyman+Normal, Sharp+Normal, and Sharp+Bootstrap to construct 95\% confidence intervals. The empirical coverage probabilities are 0.992, 0.973, and 0.977, with mean confidence interval lengths of 4.88, 4.33, and 4.53, respectively. All methods reach the nominal confidence level of 95\%, with Sharp+Normal and Sharp+Bootstrap achieving reductions in interval lengths of 11.3\% and 7.1\% compared to Neyman+Normal, respectively.

\subsection{A Public Health Field Experiment in Mexico}
\label{sec:rd:pair}

In this section, we apply the causal bootstrap based on constant-treatment-effect imputation in a paired cluster experiment conducted in Mexico \citep{king2007politically}.
This experiment aims to investigate the effect of universal health insurance program on out-of-pocket health-care expenses , examining the treatment effect at the cluster level by treating each cluster as a unit.
The clusters were paired based on pre-treatment characteristics, including census demographics, poverty, education, and health infrastructure, resulting in a total of $M=50$ pairs.
In each pair, one cluster is randomized to the treatment group, which was encouraged to enroll in the universal health insurance program, Seguro Popular de Salud (SPS), while another cluster is assigned to the control group, which did not receive such encouragement.
The primary outcome is the mean out-of-pocket healthcare expenses in the cluster. Figure~\ref{fig:SPS} demonstrates that the distribution of the outcome is skewed and contains outliers.

\begin{figure*}[!htbp]
\centerline{\includegraphics[width=\textwidth]{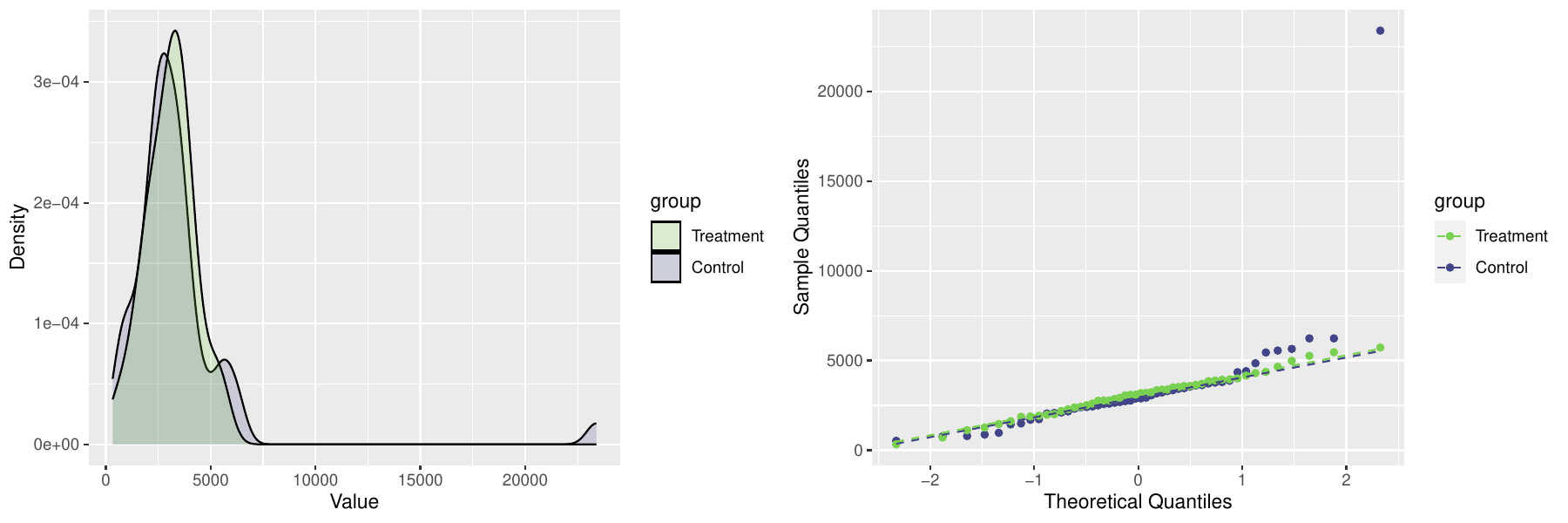}}
\caption{Density plot and Q-Q plot for the outcomes from the public health field experiment in Mexico.\label{fig:SPS}\hspace{0.5\textwidth}}
\end{figure*}

We utilize both the normal approximation and the causal bootstrap to construct 95\% confidence intervals, yielding $(-1232,561)$ and $(-1089,416)$, respectively. 
Both methods indicate that there is no significant evidence to demonstrate the effect of the universal health insurance program on out-of-pocket healthcare expenses.
Compared to the normal approximation, the causal bootstrap reduces the interval length by 16.1\%, showcasing an improvement in inference efficiency.

We assess the repeated sampling performance of the proposed methods by generating a semi-synthetic dataset and conducting simulations. The imputed population is generated based on the observed outcomes with a normal distribution. Specifically, within the $m$-th pair, $\ymaimp[1]=\zm \ymao+(1-\zm)\ymbo$, $\ymbimp[1]\sim \max[\mathcal N\{\ymaimp[1],s_{Y(1)}^2\},0]$, $\ymbimp[0]=(1-\zm)\ymao+\zm\ymbo$ and $\ymaimp[0]\sim \max[\mathcal N\{\ymbimp[0],s_{Y(0)}^2\},0]$, where $s_{Y(z)}^2$ is calculated with the observed outcomes of all strata in respective groups. This imputation method ensures that the potential outcomes are all non-negative, based on the non-negativity of medical expenses. We conduct 1000 simulations on this semi-synthetic dataset and apply both the normal approximation method and the causal bootstrap method to construct 95\% confidence intervals. Both the empirical coverage probabilities exceed 0.95, and the mean interval lengths are 1953 and 1810, respectively. In comparison to the normal approximation, the causal bootstrap reduces the interval length by 7.3\%. Overall, these findings underscore the effectiveness of the causal bootstrap method in paired randomized experiments.

\section{Conclusion and future work}
\label{sec:dis}

In this paper, we propose a sharp variance estimator and two causal bootstrap procedures in stratified randomized experiments. The first causal bootstrap method is based on rank-preserving imputation, featuring a second-order refinement over the normal approximation when there are at least two treated units and two control units in each stratum. The second causal bootstrap method is based on constant-treatment-effect imputation, making it applicable to paired experiments. Numerical results show that the proposed methods have significant advantages compared to the normal approximation with a Neyman-type variance estimator in the presence of outliers or non-normality. 

Since the sharp variance estimator and rank-preserving causal bootstrap in stratified randomization experiments require each stratum to contain at least two treated and two control units, they cannot be applied to paired experiments. In the process of preparing this paper for resubmission, we became aware of several recent papers that improve variance estimation and inference in the randomization-based inference framework \citep{mikhaeil2024sharp, chattopadhyay2024neymanian, brennan2024causal}. Although \Cref{proposition.pair} demonstrates that the commonly used variance upper bound $\sigma_{\p}^2$ is sharp within the class of unbiasedly estimable variance upper bounds, deriving sharp variance upper bound and induced causal bootstrap for paired experiments under other feasible classes remains an open problem and is left for further investigation in future work.

By leveraging covariate information, \citet{wang2020sharp} obtained a tighter sharp bound for the variance of the ATE estimator in completely randomized experiments. Investigating the development of a sharp variance estimator and causal bootstrap procedure in the presence of covariates represents a promising avenue for future research.

\citet{kosko2024fast} proposed a causal bag of little bootstraps to improve the computational efficiency of causal bootstrap methods while maintaining consistent estimates and reliable confidence interval coverage. Applying this approach to accelerate our methods would also be an interesting direction to explore.



\section*{Acknowledgments}
Dr. Liu was supported by the National Natural Science Foundation of China (12071242).
The information reported here results from secondary analyses of data from clinical trials conducted by the National Institute on Drug Abuse (NIDA). Specifically, data from NIDA-CTN-0053 (Achieving Cannabis Cessation-Evaluating N-Acetylcysteine Treatment) were included. NIDA databases and information are available at \url{http://datashare.nida.nih.gov}.
The authors would like to express their sincere gratitude to the editor, the associate editor and two reviewers for their insightful feedback and constructive suggestions, which have significantly improved the quality of this manuscript.

\section*{Supporting information}

Additional supporting information may be found in the
online version of the article at the publisher’s website.
The \texttt{R} package \texttt{CausalBootstrap} implementing our method is available at \href{https://github.com/yu-hao-yang/CausalBootstrap}{https://github.com/yu-hao-yang/CausalBootstrap}.

\bibliographystyle{apalike.bst}


\end{document}